\documentclass[reqno]{amsart}
\usepackage{amssymb}
\usepackage{hyperref}
\usepackage[T1]{fontenc}
\usepackage{dsfont}
\newtheorem{thm}{Theorem}[section]

\newtheorem{lem}{Lemma}[section]
\newtheorem{prop}{Proposition}[section]

\begin{document}
\title[ ]{ Quasi-Fredholm and Saphar spectrums for the $\alpha$-times integrated semigroups }

\author[ A. Tajmouati,  A. El Bakkali, M.B. Mohamed Ahmed and H. Boua
\\]
{ A. Tajmouati, A. El Bakkali, M.B. Mohamed Ahmed and H. Boua}
\address{A. Tajmouati, M.B. Mohamed Ahmed and H. Boua\newline
 Sidi Mohamed Ben Abdellah
 Univeristy,
 Faculty of Sciences Dhar Al Mahraz, Fez,  Morocco.}
\email{abdelaziz.tajmouati@usmba.ac.ma}
\email{bbaba2012@gmail.com}
\email{hamid12boua@yahoo.com}
\address{A. El Bakkali  \newline
Department of Mathematics
University Chouaib Doukkali,
Faculty of Sciences.
24000, Eljadida, Morocco.}
\email{aba0101q@yahoo.fr}
\maketitle
\begin{abstract}
We continue to study $\alpha$-times integrated semigroups. Essentially, we characterize the different spectrums of $\alpha$-times integrated semigroups by the spectrums of their generators. Particulary quasi-Fredholm, Kato, essentially Kato, Saphar and essentially Saphar spectrums.
\end{abstract}
\section{Introduction}
Let $X$ be a complex Banach space and $\mathcal{B}(X)$ the algebra of all bounded linear operators on  $X$. We denote by $D(T)$, $R(T)$, $R^\infty(T):=\cap_{n\geq 1}R(T^n)$, $N(T)$, $\rho(T)$, $\sigma(T),$ and $\sigma_p(T)$ respectively the domain, the range, the hyper range, the kernel, the resolvent and the spectrum of $T$, where
$\sigma(T)=\{\lambda\in\mathds{C}\,\backslash \, \lambda-T \,\mbox{is not bijective}\}$ and $\sigma_p(T)=\{\lambda\in\mathds{C}\,\backslash \, \lambda-T \,\mbox{is not one to one}\}.$ The function resolvent of $T\in\mathcal{B}(X)$ is defined for all $\lambda\in\rho(T)$ by $R(\lambda,T)=(\lambda -T)^{-1}.$
An operator $T$ is called Kato, in symbol $T\in\mathcal{D}(X)$, if $R(T)$ is closed and $N(T)\subseteq R^\infty(T)$.
An operator $T$ is called essentially Kato, in symbol $T\in e\mathcal{D}(X)$, if $R(T)$ is closed and $N(T)\subseteq_e R^\infty(T)$.
An operator $T$ is called relatively regular if there exists $S$ such that $TST=T$.\\
For the subspaces $M$ and $N$ of $X$ we write $M\subseteq_e N$ if there exists a finite-dimensional subspace $F\subseteq X$ such that $M\subseteq N +F$. We can choose $F$ satisfying $F\subseteq M$ and $F\cap N=\varnothing.$
An operator $T$ is called Saphar, in symbol $T\in\mathcal{S}(X)$, if $T$ is relatively regular and $N(T)\subseteq_e R^\infty(T)$.
An operator $T$ is called essentially Saphar, in symbol $T\in e\mathcal{S}(X)$, if $T$ is relatively regular and $N(T)\subseteq_e R^\infty(T)$.
The Kato, essentially Kato , Saphar and essentially Saphar spectrums are defined by
$$\sigma_K(T)=\{\lambda\in\mathds{C}\,\backslash\, \lambda-T\in \mathcal{D}(X)\};$$
$$\sigma_{eK}(T)=\{\lambda\in\mathds{C}\,\backslash\, \lambda-T\in e\mathcal{D}(X)\};$$
 $$\sigma_S(T)=\{\lambda\in\mathds{C}\,\backslash\, \lambda-T\in \mathcal{S}(X)\};$$
$$\sigma_{eS}(T)=\{\lambda\in\mathds{C}\,\backslash\, \lambda-T\in e\mathcal{S}(X)\};$$
The degree of stable iteration $dis(T)$ of an operator $T$ is defined by
  $$dis(T)=inf\{n\in\mathds{N}\,\backslash\,\forall m\geq n, R(T^n)\cap N(T)=R(T^m)\cap N(T)\}.$$ An operator $T$ is called quasi-Fredholm, in symbol $T\in q\Phi(X)$, if there exists $d\in\mathds{N}$ such that $R(T^n)$ and $R(T)+N(T^n)$ are closed for all $n\geq d$ and $dis(T)=d$. The quasi-Fredholm spectrum is defined by
  $$\sigma_{qe}(T)=\{\lambda\in\mathds{C}\,\backslash\, \lambda-T\notin q\Phi(X)\}.$$
 Let $\alpha\geq 0$ and let $A$ be a linear operator on a Banach space $X$. We recall that $A$ is the generator of an $\alpha$-times integrated semigroup
$(S(t))_{t\geq 0}$ on $X$ \cite{r.5} if
$]\omega, +\infty[ \subseteq \rho(A)$ for some $\omega\in \mathds{R}$
and there exists a strongly continuous mapping $S: [0, +\infty[ \rightarrow \mathcal{B}(X)$ satisfying
\begin{eqnarray*}
\|S(t)\|&\leq& Me^{\omega t} \,\,\mbox{for all}\,\, t\geq 0 \,\,\mbox{ and some }\,\, M > 0;\\
R(\lambda, A)&=& \lambda^\alpha\int_0^{+\infty} e^{-\lambda t}S(t)ds \,\,\mbox{for all}\,\, \lambda> \max\{\omega, 0\},
\end{eqnarray*}
in this case, $(S(t))_{t\geq 0}$ is called an $\alpha$-times integrated semigroup and the domain of its generator $A$ is defined by
$$D(A)=\{x\in X\, / \,\int_0^tS(s)Axds=S(t)x-\frac{t^\alpha x}{\Gamma(\alpha+1)}\},$$
where $\Gamma$ is the Euler integral giving by $$\Gamma(\alpha+1)=\int_0^{+\infty} x^\alpha e^{-x}dx.$$
We know that $(S(t))_{t\geq 0}\subseteq\mathcal{B}(X)$ is an $\alpha$-times integrated semigroup if and only if
$$S(t+s)=\frac{1}{\Gamma(\alpha)}[\int_t^{t+s}(t+s-r)^{\alpha-1}S(r)xdr-
\int_0^{s}(t+s-r)^{\alpha-1}S(r)xdr]$$ for all $x\in X$ and all $t,s\geq 0$.\\
In \cite{r.3}, the authors have studied the different spectrums of the 1-times integrated semigroups.
In our paper \cite{r.12}, we have studied descent, ascent, Drazin, Fredholm and Browder spectrums of an $\alpha$-times integrated semigroup.
Also in \cite{r.13}, we have investigated essential ascent and descent, upper and lower semi-Fredholm and semi-Browder spectrums of an $\alpha$-times integrated semigroup.
In this paper, we continue to study the $\alpha$-times integrated semigroups for all $\alpha>0$. We investigate the relationships between the different spectrums of the $\alpha$-times integrated semigroups and their generators, precisely quasi-Fredholm, Kato, essentially Kato, Saphar and essentially Saphar spectrums.

\section{Main results}

\begin{lem}\label{l0}\cite[Proposition 2.4]{r.1}
Let $A$ be the generator of an $\alpha$-times integrated semigroup $(S(t))_{t\geq 0}\subseteq \mathcal{B}(X)$ where $\alpha \geq 0.$ Then for all $x\in D(A)$ and all $t\geq 0$ we have
\begin{enumerate}
  \item $S(t)x\in D(A)$ and $AS(t)x=S(t)Ax.$
  \item $S(t)x=\frac{t^\alpha}{\Gamma(\alpha+1)}x+\int_0^t S(s)Axds.$
\end{enumerate}
Moreover, for all $x\in X$ we get $\int_0^t S(s)xds\in D(A)$ and
  $$A\int_0^t S(s)xds=S(t)x -\frac{t^\alpha}{\Gamma(\alpha+1)}x.$$
\end{lem}

We begin by the lemmas.
\begin{lem}\label{l1} Let $A$ be the generator of an $\alpha$-times integrated semigroup $(S(t))_{t\geq 0}$ with $\alpha>0$. Then for all $\lambda\in\mathds{C}$ and all $t\geq 0$
\begin{enumerate}
\item $(\lambda-A)D_\lambda(t)x=\int_0^t e^{\lambda(t-s)}\frac{s^{\alpha-1} x}{\Gamma(\alpha)}ds-S(t)x,\,\,\forall x\in X$  where $$D_\lambda(t)x=\int_0^t e^{\lambda (t-r)}S(r)dr;$$
\item $D_\lambda(t)(\lambda-A)x=\int_0^t e^{\lambda (t-s)}\frac{s^{\alpha-1} x}{\Gamma(\alpha)}ds-S(t)x,\,\,\forall x\in D(A)$.
\end{enumerate}
\end{lem}
\begin{proof}
\begin{enumerate}
 \item By Lemma \ref{l0}, we know that for all $x\in D(A)$
 $$S(s)x=\frac{s^\alpha}{\Gamma(\alpha+1)}x+\int_0^sS(r)Axdr.$$
 Then, since $\Gamma(\alpha+1)=\alpha\Gamma(\alpha)$, we obtain
 $$S'(s)x= \frac{s^{\alpha-1}}{\Gamma(\alpha)}x+S(s)Ax.$$
 Therefore, we conclude that
 \begin{eqnarray*}
 D_\lambda(t)Ax &=&\int_0^te^{\lambda(t-s)}S(s)Axds\\
  &=& \int_0^te^{\lambda(t-s)}[S'(s)x-\frac{s^{\alpha-1}}{\Gamma(\alpha)}x]ds\\
 &=& \int_0^te^{\lambda(t-s)}S'(s)xds-
 \int_0^te^{\lambda(t-s)}\frac{s^{\alpha-1}}{\Gamma(\alpha)}xds\\
 &=& S(t)x+\lambda D_\lambda(t)x-
   \int_0^te^{\lambda(t-s)}\frac{s^{\alpha-1}}{\Gamma(\alpha)}xds
     \end{eqnarray*}
Finally, we obtain for all $x\in D(A)$
$$D_\lambda(t)(\lambda-A)x= \big[\int_0^t e^{\lambda(t-s)}\frac{s^{\alpha-1}}{\Gamma(\alpha)}ds - S(t)\big]x.$$
\item  Let $\mu\in \rho(A)$. From proof of Lemma \ref{l0}, we have
for all $x\in X$ $$R(\mu,A)S(s)x=S(s)R(\mu,A)x.$$
Hence, for all $x\in X$ we conclude
\begin{eqnarray*}
R(\mu,A)D_\lambda(t)x &=& R(\mu,A)\int_0^te^{\lambda (t-s)}S(s)xds\\
   &=& \int_0^te^{\lambda (t-s)}R(\mu,A)S(s)xds\\
   &=& \int_0^te^{\lambda (t-s)}S(s)R(\mu,A)xds\\
&=& D_\lambda(t)R(\mu,A)x.
   \end{eqnarray*}
Therefore, we obtain for all $x\in X$
 \begin{eqnarray*}
 D_\lambda(t)x &=& \int_0^te^{\lambda (t-s)}S(s)xds\\
&=& \int_0^te^{\lambda (t-s)}S(s)(\mu-A)R(\mu,A)xds  \\
&=& \mu \int_0^te^{\lambda (t-s)} S(s)R(\mu,A)xds -\int_0^te^{\lambda (t-s)}S(s)AR(\mu,A)xds \\
&=& \mu \int_0^te^{\lambda (t-s)} R(\mu,A)S(s)xds -\int_0^te^{\lambda (t-s)}S(s)AR(\mu,A)xds \\
&=& \mu R(\mu,A)\int_0^te^{\lambda (t-s)} S(s)xds -\int_0^te^{\lambda (t-s)}S(s)AR(\mu,A)xds \\
&=& \mu R(\mu,A)D_\lambda(t)x-D_\lambda(t)AR(\mu,A)x\\
&=& \mu R(\mu,A)D_\lambda(t)x-[S(t)R(\mu,A)x+\lambda D_\lambda(t)R(\mu,A)x-\int_0^te^{\lambda(t-s)}\frac{s^{\alpha-1}}{\Gamma(\alpha)}R(\mu,A)xds]\\
&=& \mu R(\mu,A)D_\lambda(t)x-[R(\mu,A)S(t)x+\lambda R(\mu,A)D_\lambda(t)x-R(\mu,A)\int_0^te^{\lambda(t-s)}\frac{s^{\alpha-1}}{\Gamma(\alpha)}xds]\\ &=&  R(\mu,A)\big[(\mu-\lambda)D_\lambda(t)x-S(t)x+
\int_0^te^{\lambda(t-s)}\frac{s^{\alpha-1}}{\Gamma(\alpha)}xds\big]
\end{eqnarray*}
Therefore, for all $x\in X$ we have $D_\lambda(t)x\in D(A)$ and
$$(\mu-A)D_\lambda(t)x=(\mu-\lambda) D_\lambda(t)x+\int_0^t e^{\lambda(t-s)}\frac{s^{\alpha-1}}{\Gamma(\alpha)}xds -
S(t)x.$$
Finally, for all $x\in X$ and all $\lambda\in \mathds{C}$ we obtain
$$(\lambda-A)D_\lambda(t)x=\big[\int_0^t e^{\lambda(t-s)}\frac{s^{\alpha-1}}{\Gamma(\alpha)}ds -
S(t)\big]x.$$
\end{enumerate}
\end{proof}

\begin{lem}\label{l2} Let $A$ be the generator of an $\alpha$-times integrated semigroup $(S(t))_{t\geq 0}$ with $\alpha>0$. Then for all $\lambda\in\mathds{C}$, all $t\geq 0$ and all $x\in X$
\begin{enumerate}
\item We have the identity  $$\,\,(\lambda-A)L_\lambda(t)+\varphi_\lambda(t)D_\lambda(t)=\phi_\lambda(t)I,$$
    where $L_\lambda(t)=\int_0^t e^{-\lambda s }D_\lambda(s)ds,\, \varphi_\lambda(t)=e^{\lambda t}\,\mbox{and}\, \phi_\lambda(t)=\int_0^t\int_0^\tau e^{-\lambda r}\frac{r^{\alpha-1}}{\Gamma(\alpha)}drd\tau.$\\
Moreover, the operator $L_\lambda(t)$ is commute with each one of $D_\lambda(t)$and $(\lambda-A)$.
\item For all $n\in\mathds{N}^*,$ there exists an $L_{\lambda,n}(t)\in \mathcal{B}(X)$ such that $$(\lambda-A)L_{\lambda,n}(t)+[\varphi_\lambda(t)]^n[D_\lambda(t)]^n=
    [\phi_\lambda(t)]^nI.$$
    Moreover, the operator $L_{\lambda,n}(t)$ is commute with each one of $D_\lambda(t)$ and $\lambda-A$.
\item For all $n\in\mathds{N}^*,$ there exists an operator $D_{\lambda,n}(t)\in \mathcal{B}(X)$ such that $$(\lambda-A)^n[L_\lambda(t)]^n+D_{\lambda,n}(t)D_\lambda(t)=[\phi_\lambda(t)]^nI.$$
    Moreover, the operator $D_{\lambda,n}(t)$ is commute with each one of $D_\lambda(t)$, $L_\lambda(t)$ and $\lambda-A$.
 \item  For all $n\in\mathds{N}^*,$ there exists an operator $K_{\lambda,n}(t)\in \mathcal{B}(X)$ such that
$$(\lambda-A)^nK_{\lambda,n}(t)+[D_\lambda(t)]^n[D_{\lambda,n}(t)]^n =[\phi_\lambda(t)]^{n^2}I,$$
Moreover, the operator $K_{\lambda,n}(t)$ is commute with each one of $D_\lambda(t)$, $D_{\lambda,n}(t)$ and $\lambda-A$.
\end{enumerate}
\end{lem}
\begin{proof}
\begin{enumerate}
\item Let $\mu\in \rho(A)$. By Lemma \ref{l1}, for all $x\in X$ we have $D_\lambda(s)x\in D(A)$ and hence
 \begin{eqnarray*}
 L_\lambda(t)x &=& \int_0^te^{-\lambda s}D_\lambda(s)xds\\
 &=& \int_0^te^{-\lambda s}R(\mu,A)(\mu-A)D_\lambda(s)xds\\
 &=& R(\mu,A)[\mu\int_0^te^{-\lambda s}D_\lambda(s)xds-\int_0^te^{-\lambda s}AD_\lambda(s)xds]\\
  &=& R(\mu,A)[\mu L_\lambda(t)x-\int_0^te^{-\lambda s}AD_\lambda(s)xds]
\end{eqnarray*}
Therefore for all $x\in X$, we have $L_\lambda(t)x\in D(A)$ and
$$(\mu-A)L_\lambda(t)x=\mu L_\lambda(t)x-\int_0^te^{-\lambda s}AD_\lambda(s)xds.$$
Thus
$$AL_\lambda(t)x=\int_0^te^{-\lambda s}AD_\lambda(s)xds.$$
Hence, we conclude that
\begin{eqnarray*}
(\lambda-A)L_\lambda(t)x&=&\lambda L_\lambda(t)x -\int_0^te^{-\lambda s}AD_\lambda(s)xds\\
&=&\lambda L_\lambda(t)x -\int_0^te^{-\lambda s}\big[\lambda D_\lambda(s)x-\int_0^s e^{\lambda(s-r)}\frac{r^{\alpha-1}}{\Gamma(\alpha)}xdr +
S(s)x\big]ds\\
&=&\lambda L_\lambda(t)x -\lambda\int_0^te^{-\lambda s}D_\lambda(s)x ds +\int_0^te^{-\lambda s}\int_0^s e^{\lambda(s-r)}\frac{r^{\alpha-1}}{\Gamma(\alpha)}xdrds -\int_0^te^{-\lambda s}
S(s)xds\\
&=&\lambda L_\lambda(t)x -\lambda L_\lambda(t)x +\int_0^t\int_0^s e^{-\lambda r}\frac{r^{\alpha-1}}{\Gamma(\alpha)}xdrds -e^{-\lambda t}\int_0^te^{\lambda (t-s)}S(s)xds  \\
&=&\int_0^t\int_0^s e^{-\lambda r}\frac{r^{\alpha-1}}{\Gamma(\alpha)}xdrds -e^{-\lambda t}D_\lambda (t)x  \\
 &=& \big[\phi_\lambda(t)I-\varphi_\lambda(t)D_\lambda(t)\big]x,
\end{eqnarray*}
where $\phi_\lambda(t)=\int_0^t\int_0^s e^{-\lambda r}\frac{r^{\alpha-1}}{\Gamma(\alpha)}drds$ and $\varphi_\lambda(t)=e^{-\lambda t}.$\\
Therefore, we obtain $$(\lambda-A)L_\lambda(t)+\varphi_\lambda(t)D_\lambda(t)=\phi_\lambda(t)I.$$
Since $S(s)S(t)=S(t)S(s)$ for all $s,t\geq 0$, then
 $D_\lambda(s)S(t)=S(t)D_\lambda(s).$\\
Hence
\begin{eqnarray*}
D_\lambda(t)D_\lambda(s) &=& \int_0^te^{\lambda(t-r)}S(r)D_\lambda(s)dr\\
&=& \int_0^te^{\lambda(t-r)}S(r)D_\lambda(s)dr\\
&=& \int_0^te^{\lambda(t-r)}D_\lambda(s)S(r)dr\\
&=& D_\lambda(s)\int_0^te^{\lambda(t-r)}S(r)dr\\
&=& D_\lambda(s)D_\lambda(t).
\end{eqnarray*}
Thus, we deduce that
\begin{eqnarray*}
D_\lambda(t)L_\lambda(t) &=& D_\lambda(t)\int_0^te^{-\lambda s}D_\lambda(s)ds\\
&=& \int_0^te^{-\lambda s}D_\lambda(t)D_\lambda(s)ds\\
&=& \int_0^te^{-\lambda s}D_\lambda(s)D_\lambda(t)ds\\
&=& \int_0^te^{-\lambda s}D_\lambda(s)dsD_\lambda(t)\\
&=& L_\lambda(t)D_\lambda(t).
\end{eqnarray*}
Since for all $x\in X$ $AL_\lambda(t)x=\int_0^te^{-\lambda s}AD_\lambda(s)xds$ and for all $x\in D(A)$ $AD_\lambda(s)x=D_\lambda(s)Ax,$
then we obtain for all $x\in D(A)$
\begin{eqnarray*}
(\lambda-A)L_\lambda(t)x &=& \lambda L_\lambda(t)x-AL_\lambda(t)x\\
&=& \lambda L_\lambda(t)x -\int_0^te^{-\lambda s}AD_\lambda(s)xds\\
&=& \lambda L_\lambda(t)x -\int_0^te^{-\lambda s}AD_\lambda(s)xds\\
&=&\lambda L_\lambda(t)x -\int_0^te^{-\lambda s}D_\lambda(s)Axds\\
&=& \lambda L_\lambda(t)x- L_\lambda(t)Ax\\
&=& L_\lambda(t)(\lambda-A)x.
\end{eqnarray*}
\item Since $(\lambda-A)L_\lambda(t)+\varphi_\lambda(t)D_\lambda(t)=\phi_\lambda(t)I$, then for all $n\in \mathds{N}^*$ we obtain
\begin{eqnarray*}
[\varphi_\lambda(t)D_\lambda(t)]^n &=&[\phi_\lambda(t)I-(\lambda-A)L_\lambda(t)]^n\\
&=&\sum_{i=0}^n C_n^i[\phi_\lambda(t)]^{n-i}[-(\lambda-A)L_\lambda(t)]^i\\
&=& [\phi_\lambda(t)]^nI -(\lambda-A)\sum_{i=1}^n C_n^i[\phi_\lambda(t)]^{n-i}[-(\lambda-A)]^{i-1}[L_\lambda(t)]^i\\
&=& [\phi_\lambda(t)]^nI -(\lambda-A)L_{\lambda,n}(t),
\end{eqnarray*}
where $$L_{\lambda,n}(t)=\sum_{i=1}^nC_n^i [\phi_\lambda(t)]^{n-i}[-(\lambda-A)]^{i-1}[L_\lambda(t)]^i.$$
Therefore, we have
$$(\lambda-A)L_{\lambda,n}(t)+[\varphi_\lambda(t)]^n[D_\lambda(t)]^n =[\phi_\lambda(t)]^nI.$$
Finally, for commutativity, it is clear that $L_{\lambda,n}(t)$ commute with each one of $D_\lambda(t)$ and $\lambda-A$.
\item For all $n\in \mathds{N}^*$, we obtain
\begin{eqnarray*}
[(\lambda-A)L_\lambda(t)]^n &=&[\phi_\lambda(t)I-\varphi_\lambda(t)D_\lambda(t)]^n\\
&=&\sum_{i=0}^n C_n^i[\phi_\lambda(t)]^{n-i}[-\varphi_\lambda(t)D_\lambda(t)]^i\\
&=& [\phi_\lambda(t)]^nI -D_\lambda(t)\sum_{i=1}^nC_n^i [\phi_\lambda(t)]^{n-i}[\varphi_\lambda(t)]^{i}[-D_\lambda(t)]^{i-1}\\
&=& [\phi_\lambda(t)]^nI -D_\lambda(t)D_{\lambda,n}(t),
\end{eqnarray*}
where $$D_{\lambda,n}(t)=\sum_{i=1}^n C_n^i [\phi_\lambda(t)]^{n-i}[\varphi_\lambda(t)]^{i}[-D_\lambda(t)]^{i-1}.$$
Therefore, we have
$$(\lambda-A)^n[L_\lambda(t)]^n+D_\lambda(t)D_{\lambda,n}(t)=[\phi_\lambda(t)]^nI.$$
Finally, for commutativity, it is clear that $D_{\lambda,n}(t)$ commute with each one of $D_\lambda(t)$, $L_\lambda(t)$ and $\lambda-A$.
\item
Since we have
$D_\lambda(t)D_{\lambda,n}(t)=[\phi_\lambda(t)]^nI-(\lambda-A)^n[L_\lambda(t)]^n,$
then for all $n\in \mathds{N}$
\begin{eqnarray*}
[D_\lambda(t)D_{\lambda,n}(t)]^n &=& \big[[\phi_\lambda(t)]^nI-(\lambda-A)^n[L_\lambda(t)]^n\big]^n\\
 &=& [\phi_\lambda(t)]^{n^2}I-\sum_{i=1}^n C_n^i \big[[\phi_\lambda(t)]^{n}\big]^{n-i}\big[(\lambda-A)^n[L_\lambda(t)]^n\big]^i\\
 &=& [\phi_\lambda(t)]^{n^2}I-(\lambda-A)^n\sum_{i=1}^nC_n^i \big[[\phi_\lambda(t)]^{n(n-i)}(\lambda-A)^{n(i-1)}[L_\lambda(t)]^{ni}\\
&=& [\phi_\lambda(t)]^{n^2}I-(\lambda-A)^nK_{\lambda,n}(t),
\end{eqnarray*}
where $K_{\lambda,n}(t)=\sum_{i=1}^nC_n^i
[\phi_\lambda(t)]^{n(n-i)}(\lambda-A)^{n(i-1)}[L_\lambda(t)]^{ni}.$
Hence we obtain
$$[D_\lambda(t)]^n[D_{\lambda,n}(t)]^n +(\lambda-A)^nK_{\lambda,n}(t) =[\phi_\lambda(t)]^{n^2}I.$$
Finally, the commutativity is clear.
\end{enumerate}
\end{proof}

Now, we prove this result.
\begin{prop}\label{p1} Let $A$ be the generator of an $\alpha$-times integrated semigroup
$(S(t))_{t\geq 0}$ with $\alpha>0$. For all $\lambda\in\mathds{C}$ and all $t\geq 0$, if $R[\int_0^te^{\lambda (t-s)}\frac{s^{\alpha-1}}{\Gamma(\alpha)}ds-S(t)]^n$ is closed, then $R(\lambda-A)^n$ is also closed.
\end{prop}

\begin{proof}
Let $(y_n)_{n\in\mathds{N}}\subseteq X$ such that $y_n\rightarrow y\in X$ and there exists $(x_n)_{n\in\mathds{N}}\subseteq D(A)$ satisfying
      $$(\lambda-A)^mx_n=y_n.$$
      By Lemma \ref{l2}, we obtain
      $$(\lambda-A)^m[L_\lambda(t)]^my_n+G_{\lambda,m}(t)D_\lambda(t)y_n=[\phi_\lambda(t)]^my_n.$$
Hence, we have
\begin{eqnarray*}
[\int_0^te^{\lambda (t-s)}\frac{s^{\alpha-1}}{\Gamma(\alpha)}ds-S(t)]^mG_{\lambda,m}(t)x_n
&=& [D_\lambda(t)]^m(\lambda -A)^mG_{\lambda,m}(t)x_n;\\
&=& G_{\lambda,m}(t)[D_\lambda(t)]^m(\lambda -A)^mx_n;\\
&=& G_{\lambda,m}(t)[D_\lambda(t)]^my_n;\\
&=& [\phi_\lambda(t)]^my_n-(\lambda-A)^m[L_\lambda(t)]^my_n.
\end{eqnarray*}
Then, $$[\phi_\lambda(t)]^my_n-(\lambda-A)L_\lambda(t)y_n \in R[\int_0^te^{\lambda (t-s)}\frac{s^{\alpha-1}}{\Gamma(\alpha)}ds-S(t)]^m.$$
Since $R[\int_0^te^{\lambda (t-s)}\frac{s^{\alpha-1}}{\Gamma(\alpha)}ds-S(t)]^m$ is closed, hence $G_{\lambda,m}(t)$ is bounded linear and
$[\phi_\lambda(t)]^my_n-(\lambda-A)^m[L_\lambda(t)]^my_n$ converges to  $[\phi_\lambda(t)]^my-(\lambda-A)^m[L_\lambda(t)]^my$.\\
Therefore, we conclude that
$$[\phi_\lambda(t)]^my-(\lambda-A)^m[L_\lambda(t)]^my\in R[\int_0^te^{\lambda (t-s)}\frac{s^{\alpha-1}}{\Gamma(\alpha)}ds-S(t)]^m.$$
Then, there exists $z\in X$ such that
$$[\int_0^te^{\lambda (t-s)}\frac{s^{\alpha-1}}{\Gamma(\alpha)}ds-S(t)]^mz=
[\phi_\lambda(t)]^my-(\lambda-A)^m[L_\lambda(t)]^my.$$
Hence for all $t\neq 0$, we have $\phi_\lambda(t)\neq 0$ and
\begin{eqnarray*}
y &=& \frac{1}{[\phi_\lambda(t)]^m}[[\int_0^te^{\lambda (t-s)}\frac{s^{\alpha-1}}{\Gamma(\alpha)}ds-S(t)]^mz+(\lambda-A)^m[L_\lambda(t)]^my]\\
&=&\frac{1}{[\phi_\lambda(t)]^m}[(\lambda-A)^m[D_\lambda(t)]^mz+(\lambda-A)^m[L_\lambda(t)]^my]\\
&=& \frac{1}{[\phi_\lambda(t)]^m}(\lambda-A)^m[[D_\lambda(t)]^mz+[L_\lambda(t)]^my].
\end{eqnarray*}
Finally, we obtain $$y\in R(\lambda-A)^m.$$
\end{proof}

\begin{prop}\label{p2}Let $A$ be the generator of an $\alpha$-times integrated semigroup $(S(t))_{t\geq 0}$ with $\alpha>0$. Then for all $\lambda\in\mathds{C}$ and all $t\geq 0,$ we have
\begin{enumerate}
\item If $\int_0^te^{\lambda (t-s)}\frac{s^{\alpha-1}}{\Gamma(\alpha)}ds-S(t)$ is relatively regular, then $\lambda-A$ is also;
\item If $N[\int_0^te^{\lambda (t-s)}\frac{s^{\alpha-1}}{\Gamma(\alpha)}ds-S(t)]\subseteq R^\infty[\int_0^te^{\lambda (t-s)}\frac{s^{\alpha-1}}{\Gamma(\alpha)}ds-S(t)],$ then $$N(\lambda-A)\subseteq R^\infty(\lambda-A).$$
\item  If $N[\int_0^te^{\lambda (t-s)}\frac{s^{\alpha-1}}{\Gamma(\alpha)}ds-S(t)]\subseteq_e R^\infty[\int_0^te^{\lambda (t-s)}\frac{s^{\alpha-1}}{\Gamma(\alpha)}ds-S(t)],$ then $$N(\lambda-A)\subseteq_e R^\infty(\lambda-A).$$
\end{enumerate}
 \end{prop}

\begin{proof}
\begin{enumerate}
\item Suppose that
$$[\int_0^te^{\lambda (t-s)}\frac{s^{\alpha-1}}{\Gamma(\alpha)}ds-S(t)]T(t)[\int_0^te^{\lambda (t-s)}\frac{s^{\alpha-1}}{\Gamma(\alpha)}ds-S(t)]=$$
$$[\int_0^te^{\lambda (t-s)}\frac{s^{\alpha-1}}{\Gamma(\alpha)}ds-S(t)].$$
Using Lemma \ref{l2}, we obtain
\begin{eqnarray*}
\phi_\lambda(t)(\lambda-A) &=& [(\lambda-A)L_\lambda(t)+G_\lambda(t)D_\lambda(t)](\lambda-A);\\
&=& (\lambda-A)L_\lambda(t)(\lambda-A)+\psi_\lambda(t)D_\lambda(t)(\lambda-A);\\
&=& (\lambda-A)L_\lambda(t)(\lambda-A)+\psi_\lambda(t)[\int_0^te^{\lambda (t-s)}\frac{s^{\alpha-1}}{\Gamma(\alpha)}ds-S(t)];\\
&=& (\lambda-A)L_\lambda(t)(\lambda-A)\\
 &+&\psi_\lambda(t)[[\int_0^te^{\lambda (t-s)}\frac{s^{\alpha-1}}{\Gamma(\alpha)}ds-S(t)]T(t)[\int_0^te^{\lambda (t-s)}\frac{s^{\alpha-1}}{\Gamma(\alpha)}ds-S(t)]];\\
&=& (\lambda-A)L_\lambda(t)(\lambda-A)+
\psi_\lambda(t)\big[[\lambda-A)D_\lambda(t)]T(t)[D_\lambda(t)(\lambda-A)]\big];\\
&=& (\lambda-A)[L_\lambda(t)+\psi_\lambda(t)D_\lambda(t)T(t)D_\lambda(t)](\lambda-A).
\end{eqnarray*}
Therefore, $\lambda-A$ is relatively regular.
 \item It is automatic by
 \begin{eqnarray*}
 N(\lambda-A) &\subseteq& N[\int_0^te^{\lambda (t-s)}\frac{s^{\alpha-1}}{\Gamma(\alpha)}ds-S(t)];\\
&\subseteq & R^\infty[\int_0^te^{\lambda (t-s)}\frac{s^{\alpha-1}}{\Gamma(\alpha)}ds-S(t)];\\
&\subseteq & R^\infty(\lambda-A).
\end{eqnarray*}
\item It is automatic by
 \begin{eqnarray*}
 N(\lambda-A) &\subseteq& N[\int_0^te^{\lambda (t-s)}\frac{s^{\alpha-1}}{\Gamma(\alpha)}ds-S(t)];\\
&\subseteq_e & R^\infty[\int_0^te^{\lambda (t-s)}\frac{s^{\alpha-1}}{\Gamma(\alpha)}ds-S(t)];\\
&\subseteq & R^\infty(\lambda-A).
\end{eqnarray*}
\end{enumerate}
 \end{proof}

The following result discusses the Kato and Saphar spectrum.
 \begin{thm} Let $A$ be the generator of an $\alpha$-times integrated semigroup $(S(t))_{t\geq 0}$ with $\alpha>0$. Then for all $t\geq 0$
\begin{enumerate}
\item $\int_0^t e^{(t-s)\sigma_{K}(A)}\frac{s^{\alpha-1}}{\Gamma(\alpha)}ds\subseteq \sigma_{K}(S(t));$
\item $\int_0^t e^{(t-s)\sigma_{S}(A)}\frac{s^{\alpha-1}}{\Gamma(\alpha)}ds\subseteq \sigma_{S}(S(t));$
\item $\int_0^t e^{(t-s)\sigma_{eK}(A)}\frac{s^{\alpha-1}}{\Gamma(\alpha)}ds\subseteq \sigma_{eK}(S(t));$
\item $\int_0^t e^{(t-s)\sigma_{eS}(A)}\frac{s^{\alpha-1}}{\Gamma(\alpha)}ds\subseteq \sigma_{eS}(S(t)).$
\end{enumerate}
\end{thm}

\begin{proof}
\begin{enumerate}
  \item Suppose that $\int_0^t e^{\lambda(t-s)}\frac{s^{\alpha-1}}{\Gamma(\alpha)}ds\notin\sigma_K(S(t))$, then we have\\
$R[\int_0^t e^{\lambda(t-s)}\frac{s^{\alpha-1}}{\Gamma(\alpha)}ds-S(t)]$ is closed and
$$N[\int_0^t e^{\lambda(t-s)}\frac{s^{\alpha-1}}{\Gamma(\alpha)}ds-S(t)]\subseteq R^\infty[\int_0^t e^{\lambda(t-s)}\frac{s^{\alpha-1}}{\Gamma(\alpha)}ds-S(t)].$$
Thus by Propositions \ref{p1} and \ref{p2}, we obtain
$R(\lambda-A)$ is closed and $$N(\lambda-A)\subseteq R^\infty(\lambda-A).$$
Therefore $\lambda-A$ is Kato and hence $$\lambda\notin\sigma_K(A).$$
\item Suppose that $\int_0^t e^{\lambda(t-s)}\frac{s^{\alpha-1}}{\Gamma(\alpha)}ds\notin\sigma_S(S(t))$, then we have\\
 $\int_0^t e^{\lambda(t-s)}\frac{s^{\alpha-1}}{\Gamma(\alpha)}ds-S(t)$ is relatively regular and $$N[\int_0^t e^{\lambda(t-s)}\frac{s^{\alpha-1}}{\Gamma(\alpha)}ds-S(t)]\subseteq R^\infty[\int_0^t e^{\lambda(t-s)}\frac{s^{\alpha-1}}{\Gamma(\alpha)}ds-S(t)].$$
Thus by Propositions \ref{p1} and \ref{p2}, we obtain
$\lambda-A$ is relatively regular and $$N(\lambda-A)\subseteq R^\infty(\lambda-A).$$
Therefore $\lambda-A$ is Saphar and hence $$\lambda\notin\sigma_S(A).$$
  \item Suppose that $\int_0^t e^{\lambda(t-s)}\frac{s^{\alpha-1}}{\Gamma(\alpha)}ds\notin\sigma_{eK}(S(t))$, then we have\\
$R[\int_0^t e^{\lambda(t-s)}\frac{s^{\alpha-1}}{\Gamma(\alpha)}ds-S(t)]$ is closed and
$$N[\int_0^t e^{\lambda(t-s)}\frac{s^{\alpha-1}}{\Gamma(\alpha)}ds-S(t)]\subseteq_e R^\infty[\int_0^t e^{\lambda(t-s)}\frac{s^{\alpha-1}}{\Gamma(\alpha)}ds-S(t)].$$
Thus by Propositions \ref{p1} and \ref{p2}, we obtain
$R(\lambda-A)$ is closed and $$N(\lambda-A)\subseteq_e R^\infty(\lambda-A).$$
Therefore $\lambda-A$ is essentially Kato and hence $$\lambda\notin\sigma_{eK}(A).$$
\item Suppose that $\int_0^t e^{\lambda(t-s)}\frac{s^{\alpha-1}}{\Gamma(\alpha)}ds\notin\sigma_{eS}(S(t))$, then we have\\
 $\int_0^t e^{\lambda(t-s)}\frac{s^{\alpha-1}}{\Gamma(\alpha)}ds-S(t)$ is relatively regular and $$N[\int_0^t e^{\lambda(t-s)}\frac{s^{\alpha-1}}{\Gamma(\alpha)}ds-S(t)]\subseteq_e R^\infty[\int_0^t e^{\lambda(t-s)}\frac{s^{\alpha-1}}{\Gamma(\alpha)}ds-S(t)].$$
Thus by Propositions \ref{p1} and \ref{p2}, we obtain
$\lambda-A$ is relatively regular and $$N(\lambda-A)\subseteq_e R^\infty(\lambda-A).$$
Therefore $\lambda-A$ is essentially Saphar and hence $$\lambda\notin\sigma_{eS}(A).$$
\end{enumerate}
\end{proof}

\begin{prop}\label{p4} Let $A$ be the generator of an $\alpha$-times integrated semigroup $(S(t))_{t\geq 0}$ with $\alpha>0$. Then for all $\lambda\in\mathds{C}$ and all $t\geq 0$, we have\\
$dis[\int_0^te^{\lambda (t-s)}\frac{s^{\alpha-1}}{\Gamma(\alpha)}ds-S(t)]=n,$ then $dis(A-\lambda)\leq n.$
\end{prop}

\begin{proof}
Since $dis[\int_0^te^{\lambda (t-s)}\frac{s^{\alpha-1}}{\Gamma(\alpha)}ds-S(t)]=n$, then for all $m\geq n$, we have\\
    $$R[\int_0^te^{\lambda (t-s)}\frac{s^{\alpha-1}}{\Gamma(\alpha)}ds-S(t)]^m\cap N[\int_0^te^{\lambda (t-s)}\frac{s^{\alpha-1}}{\Gamma(\alpha)}ds-S(t)]=$$
    $$\quad\quad\quad\quad R[\int_0^te^{\lambda (t-s)}\frac{s^{\alpha-1}}{\Gamma(\alpha)}ds-S(t)]^n\cap N[\int_0^te^{\lambda (t-s)}\frac{s^{\alpha-1}}{\Gamma(\alpha)}ds-S(t)].$$
    Let $m\geq n$ and $y\in R(\lambda-A)^m\cap N(\lambda-A)$, then there exists $x\in X$ such that  $$y=(\lambda-A)^mx.$$
    Using Lemma \ref{l2} and since $y\in N(\lambda-A)$, we obtain
    \begin{eqnarray*}
    [\phi_\lambda(t)]^my &=& [\phi_\lambda(t)]^my\\
            &=&(\lambda-A)L_{\lambda,m}(t)y+
    [\varphi_\lambda(t)]^m[D_\lambda(t)]^my\\
    &=&L_{\lambda,m}(t)(\lambda-A)y+
    [\varphi_\lambda(t)]^m[D_\lambda(t)]^m(\lambda-A)^mx\\
    &=& [\varphi_\lambda(t)]^m[D_\lambda(t)(\lambda-A)]^mx\\
    &=& [\varphi_\lambda(t)]^m[\int_0^te^{\lambda (t-s)}\frac{s^{\alpha-1}}{\Gamma(\alpha)}ds-S(t)]^mx\\
     &=& [\int_0^te^{\lambda (t-s)}\frac{s^{\alpha-1}}{\Gamma(\alpha)}ds-S(t)]^m[\varphi_\lambda(t)]^mx.
    \end{eqnarray*}
    Then $$y\in R[\int_0^te^{\lambda (t-s)}\frac{s^{\alpha-1}}{\Gamma(\alpha)}ds-S(t)]^m.$$
    Hence, since $y\in N(\lambda-A)\subseteq N[\int_0^te^{\lambda (t-s)}\frac{s^{\alpha-1}}{\Gamma(\alpha)}ds-S(t)],$ then
    $$y\in R[\int_0^te^{\lambda (t-s)}\frac{s^{\alpha-1}}{\Gamma(\alpha)}ds-S(t)]^m\cap N[\int_0^te^{\lambda (t-s)}\frac{s^{\alpha-1}}{\Gamma(\alpha)}ds-S(t)].$$
    Therefore $$y\in R[\int_0^te^{\lambda (t-s)}\frac{s^{\alpha-1}}{\Gamma(\alpha)}ds-S(t)]^n\cap N[\int_0^te^{\lambda (t-s)}\frac{s^{\alpha-1}}{\Gamma(\alpha)}ds-S(t)].$$
    Then there exists $z\in X$ satisfying
   $$ y=[\int_0^te^{\lambda (t-s)}\frac{s^{\alpha-1}}{\Gamma(\alpha)}ds-S(t)]^nz=(\lambda-A)^n[D_\lambda(t)]^nz.$$
    So $y\in R(\lambda-A)^n,$ and therefore $$dis(\lambda-A)\leq n.$$
\end{proof}

\begin{prop}\label{p5} Let $A$ be the generator of an $\alpha$-times integrated semigroup $(S(t))_{t\geq 0}$ with $\alpha>0$. For all $\lambda\in\mathds{C}$ and all $t\geq 0$, if $R[\int_0^te^{\lambda (t-s)}\frac{s^{\alpha-1}}{\Gamma(\alpha)}ds-S(t)]+N[\int_0^te^{\lambda (t-s)}\frac{s^{\alpha-1}}{\Gamma(\alpha)}ds-S(t)]^n$ is closed, then $R(\lambda-A)+N(\lambda-A)^n$ is also.
\end{prop}

\begin{proof}
Let $(y_n)_{n\in\mathds{N}}\subseteq X$ such that $y_n\rightarrow y\in X$ and there exist $(x_n)_{n\in\mathds{N}}\subseteq D(A)$ and $(z_n)_{n\in\mathds{N}}\subseteq N(\lambda-A)^m$ satisfying
      $$y_n=(\lambda-A)x_n+z_n.$$
      By Lemma \ref{l2}, we obtain
\begin{eqnarray*}
[D_\lambda(t)]^my_n &=& [D_\lambda(t)]^m(\lambda-A)x_n+[D_\lambda(t)]^mz_n;\\
         &=& [\int_0^te^{\lambda (t-s)}\frac{s^{\alpha-1}}{\Gamma(\alpha)}ds-S(t)][D_\lambda(t)]^{m-1}x_n+
         [D_\lambda(t)]^mz_n;\\
\end{eqnarray*}
Since $$[\int_0^te^{\lambda (t-s)}\frac{s^{\alpha-1}}{\Gamma(\alpha)}ds-S(t)]^m[D_\lambda(t)]^mz_n=
[D_\lambda(t)]^m[D_\lambda(t)]^m(\lambda-A)^mz_n=0,$$
we conclude that $$[D_\lambda(t)]^my_n\in R[\int_0^te^{\lambda (t-s)}\frac{s^{\alpha-1}}{\Gamma(\alpha)}ds-S(t)]+N[\int_0^te^{\lambda (t-s)}\frac{s^{\alpha-1}}{\Gamma(\alpha)}ds-S(t)]^m.$$
Moreover, we have $$R[\int_0^te^{\lambda (t-s)}\frac{s^{\alpha-1}}{\Gamma(\alpha)}ds-S(t)]+N[\int_0^te^{\lambda (t-s)}\frac{s^{\alpha-1}}{\Gamma(\alpha)}ds-S(t)]^m$$ is closed and $[D_\lambda(t)]^my_n$ converges to $[D_\lambda(t)]^my$, then there exist
$x\in X$ and $z\in N[\int_0^te^{\lambda (t-s)}\frac{s^{\alpha-1}}{\Gamma(\alpha)}ds-S(t)]^m$ such that $$[D_\lambda(t)]^my=[\int_0^te^{\lambda (t-s)}\frac{s^{\alpha-1}}{\Gamma(\alpha)}ds-S(t)]x+z.$$
Hence, we have
\begin{eqnarray*}
[D_\lambda(t)]^{2m}y &=& [D_\lambda(t)]^m[D_\lambda(t)]^my;\\
&=& [D_\lambda(t)]^m[\int_0^te^{\lambda (t-s)}\frac{s^{\alpha-1}}{\Gamma(\alpha)}ds-S(t)]x+[D_\lambda(t)]^mz,
\end{eqnarray*}
Therefore, using Lemma \ref{l2}, we obtain
\begin{eqnarray*}
[\phi_\lambda(t)]^{2m}y &=& (\lambda-A)L_{\lambda,2m}(t)y+
[\phi_\lambda(t)]^{2m}[\varphi_{\lambda}(t)]^{2m}[D_{\lambda}(t)]^{2m}y;\\
 &=& (\lambda-A)L_{\lambda,2m}(t)y+
 [\varphi_{\lambda}(t)]^{2m}[[D_\lambda(t)]^m[\int_0^te^{\lambda (t-s)}\frac{s^{\alpha-1}}{\Gamma(\alpha)}ds-S(t)]x+[D_\lambda(t)]^mz];\\
 &=& (\lambda-A)L_{\lambda,2m}(t)y+[\varphi_{\lambda}(t)]^{2m}
 [[D_\lambda(t)]^mD_\lambda(t)(\lambda-A)x+[D_\lambda(t)]^mz];\\
 &=& (\lambda-A)[L_{\lambda,2m}(t)y+
 [\varphi_{\lambda}(t)]^{2m}
 [D_\lambda(t)]^{m+1}x]+\varphi_{\lambda}(t)]^{2m}[D_\lambda(t)]^mz.
\end{eqnarray*}
Since  $$(\lambda-A)^m[\varphi_{\lambda}(t)^{2m}[D_\lambda(t)]^mz]=\varphi_{\lambda}(t)^{2m}[\int_0^te^{\lambda (t-s)}\frac{s^{\alpha-1}}{\Gamma(\alpha)}ds-S(t)]^mz=0,$$
we deduce that $$y\in R(\lambda-A)+N(\lambda-A)^m.$$
\end{proof}

The following theorem examines the quasi-Fredholm spectrum.
\begin{thm} Let $A$ be the generator of an $\alpha$-times integrated semigroup $(S(t))_{t\geq 0}$ with $\alpha>0$. Then for all $t\geq 0$, we have
 $$\int_0^t e^{(t-s)\sigma_{qe}(A)}\frac{s^{\alpha-1}}{\Gamma(\alpha)}ds\subseteq \sigma_{qe}(S(t)).$$
\end{thm}

\begin{proof}
Suppose that  $$\int_0^t e^{\lambda(t-s)}\frac{s^{\alpha-1}}{\Gamma(\alpha)}ds\notin \sigma_{qe}(S(t)).$$
Then there exists $d\in\mathds{N}$ such that for all $n\geq d$
$R[\int_0^t e^{\lambda(t-s)}\frac{s^{\alpha-1}}{\Gamma(\alpha)}ds-S(t)]^n$ and $R[\int_0^t e^{\lambda(t-s)}\frac{s^{\alpha-1}}{\Gamma(\alpha)}ds-S(t)]+N[\int_0^t e^{\lambda(t-s)}\frac{s^{\alpha-1}}{\Gamma(\alpha)}ds-S(t)]^n$ are closed
and $$dis[\int_0^t e^{\lambda(t-s)}\frac{s^{\alpha-1}}{\Gamma(\alpha)}ds-S(t)]=d.$$
Using Propositions \ref{p1}, \ref{p4} and \ref{p5}, we obtain for all $n\geq d$
 $R[\lambda-A]^n$ and $R[\lambda-A]+N[\lambda-A]^n$ are closed and $dis(\lambda-A)\leq d$.
 Therefore, $\lambda-A$ is quasi-Fredholm and hence
 $$\lambda\notin\sigma_{qe}(A).$$
\end{proof}

\end{document}